\documentclass[11pt]{article}

\usepackage{amssymb,amsmath,amsthm}
\usepackage[all]{xy}
\usepackage{url}
  \usepackage{color}
\usepackage{tikz}
  \usepackage{xy}
 \xyoption{curve}
 \xyoption{color}
 \xyoption{line}
\xyoption{arc}
\usepackage{graphicx}
\usepackage[
            pdftex,
           colorlinks=true,
            urlcolor=blue,       
            filecolor=blue,     
            linkcolor=blue,       
            citecolor=blue,         
            pdftitle= {Title},
            pdfauthor={Author},
            pdfsubject={Subject},
            pdfkeywords={Key}
            pagebackref,
            pdfpagemode=None,
            bookmarksopen=true]
            {hyperref}
\usepackage{hyperref}
\usepackage{color}
\usepackage[author=,english,status=draft]{fixme}
\usepackage{ifdraft}

\ifdraft{
\usepackage{fancyhdr}
\pagestyle{fancy}
\fancyhead{}
\fancyfoot{}
\fancyhead[CO,CE]{Draft --- \today}
\fancyhead[RO, LE] {\thepage}
}

\date{}

 \newtheorem{thm}{Theorem}[section]
 \newtheorem{prop}[thm]{Proposition}
 \newtheorem{lemma}[thm]{Lemma}
 \newtheorem{cor}[thm]{Corollary}
 
\newtheorem{question}[thm]{Question}
 
 \theoremstyle{definition}

 \newtheorem{remark}[thm]{Remark}
\numberwithin{equation}{section}

\newcommand{\bbC}{{\mathbb{C}}}

\newcommand{\bbH}{{\mathbb{H}}}

\newcommand{\bbR}{{\mathbb{R}}}
\newcommand{\bbP}{{\mathbb{P}}}

\newcommand{\bbQ}{{\mathbb{Q}}}

\newcommand{\bbZ}{{\mathbb{Z}}}

\newcommand{\bbk}{{\Bbbk}}

\newcommand{\Aut}{\operatorname{Aut}}

\newcommand{\st}{\operatorname{st}}

\newcommand{\rank}{\operatorname{rank}}

\newcommand{\id}{\operatorname{id}}

\newcommand{\Pic}{\operatorname{Pic}}

\newcommand{\half}{\frac{1}{2}}
\newcommand{\cha}{\operatorname{char}}

\newcommand{\bsm}{\left(\begin{smallmatrix}}
\newcommand{\esm}{\end{smallmatrix}\right)}

\newcommand{\la}{\langle}
\newcommand{\ra}{\rangle}
\newcommand{\da}{\dasharrow }

\newcommand{\calR}{\mathcal{R}}
\newcommand{\calS}{\mathcal{S}}

\newcommand{\calX}{\mathcal{X}}

\newcommand{\calD}{\mathcal{D}}

\newcommand{\frakS}{\mathfrak{S}}

\newcommand{\sfE}{\mathsf{E}}

\newcommand{\Or}{\operatorname{O}}
\newcommand{\Num}{\operatorname{Num}}

\newcommand{\UC}{\operatorname{UC}}

\newcommand{\beq}{\begin{equation}}
\newcommand{\eeq}{\end{equation}}

\begin{document}

\title{Salem numbers and  Enriques surfaces}
\author{}
\author{Igor Dolgachev}
\maketitle

\ifdraft{\thispagestyle{fancy}}{}
\begin{abstract} It is known that the dynamical degree of an automorphism $g$ of an algebraic surface $S$  is lower semi-continuous when $(S,g)$ varies in an algebraic family. In this paper we report on computer experiments confirming this behavior with the aim to realize small Salem numbers as the dynamical degrees of automorphisms of Enriques surfaces or rational Coble surfaces. 
\end{abstract}

\section{Introduction} Let $X$ be a smooth projective algebraic surface over an algebraically closed field $\bbk$ and $\Num(X)$ be the numerical lattice of $X$, the quotient of the Picard group $\Pic(X)$ modulo numerical equivalence.\footnote{Over $\bbC$, the group $\Num(X)$ is isomorphic to the subgroup of $H^2(X,\bbZ)$ modulo torsion generated by algebraic cycles.}  An automorphism $g$ of $X$ acts on $\Num(X)$ preserving the inner product defined by the intersection product on $\Num(X)$. It is known that the characteristic polynomial of $g^*:\Num(X)\to \Num(X)$ is the product of cyclotomic polynomials and at most one \emph{Salem polynomial}, a monic irreducible reciprocal polynomial in $\bbZ[x]$ which has two reciprocal positive roots and all other roots are complex numbers of absolute value  one (see \cite{McMul1}). 

The spectral radius $\lambda(g)$ of $g^*$, i.e. the eigenvalue of $g^*$ on $\Num(X)_\bbC$ with maximal absolute value, is equal to 1 or the real eigenvalue  larger than 1. The number $\lambda(g)$ is called the \emph{dynamical degree} of $g$. 
It expresses the growth of the degrees of iterates $g^n$ of $g$. More precisely, we have \cite{Cantat} 
$$\lambda(g) = \lim_{n\to \infty}(\deg_hg^n)^{1/n},$$
where $\deg_hg^n = (g^*)^{-n}(h)\cdot h$ is the \emph{degree} of $(g^*)^{-n}(h)$ with respect to the numerical class of an ample divisor $h$ on $X$. The dynamical degree of $g$ does not depend on a choice of $h$. 

An automorphism $g$ is called \emph{hyperbolic} if $\lambda(g) > 1$. Equivalently, in the action of $g^*$ on the hyperbolic space associated with the real Minkowski space $\Num(X)_\bbR$, $g^*$ is a hyperbolic isometry. Its two fixed points lying in the boundary correspond to the eigenvectors of $g^*$ with eigenvalues $\lambda(g)$ and $1/\lambda(g)$. The isometry $g^*$ acts on the geodesic with ends at the fixed points as a hyperbolic translation with the hyperbolic distance $\lambda(g)$.  It is known that, over $\bbC$,  $\log(\lambda(g))$ is equal to the \emph{topological entropy} of the automorphisms $g$ acting on the set of complex points equipped with the euclidean topology. 

If $\lambda(g) = 1$, then the isometry $g^*$ is elliptic or  parabolic. In the former case, $g$ is an automorphism with some power lying in the connected component of the identity of the automorphism group of $X$. In this case $\deg_hg^n$ is bounded.  In the latter case $\deg_hg^n$ grows linearly or quadratically and $g$ preserves a pencil of rational or arithmetic  genus one curves on $X$ (see \cite{Cantat2}, \cite{Gizatullin}).

Going through  the classification of algebraic surfaces, one easily checks that a hyperbolic automorphism can be realized only on surfaces birationally isomorphic to an abelian surface, a K3 surface,  an Enriques surface, or the projective plane (see \cite{Cantat2}). 

The smallest known Salem number is the \emph{Lehmer number} $\alpha_{\textrm{Leh}}$ with the minimal  polynomial
$$x^{10}+x^9-(x^3+\cdots+x^7)+x+1.$$
It is equal to $1.17628...$. It is conjectured that this is indeed the smallest Salem number. In fact, it is the smallest possible dynamical degree of an automorphism of an algebraic surface \cite{McMul3}.  The Lehmer number can be realized as the dynamical degree of an automorphism of a rational surface or a K3 surface (see \cite{McMul3}, \cite{McMul4}). On the other hand, it is known that it cannot be realized on an Enriques surface \cite{Oguiso}.

In this paper we attempt to construct hyperbolic automorphisms of Enriques surfaces of small  dynamical degree. We succeeded in realizing the second smallest Salem number of degree 2 and the third smallest Salem number in degree 4. However, we believe that our smallest Salem numbers of degrees 6,8 and 10 are far from being  minimal, so  the paper should be viewed as a computer experiment. 

The main idea for the search of hyperbolic automorphisms of small dynamical degree is based on the following  nice result of Junyi Xie \cite{Xie} that roughly says that the dynamical degree of an automorphism does not increase when the surface   together with the automorphism is specialized in an algebraic family. More precisely, we have the following theorem.

\begin{thm}\label{xie} Let $f:\calX\to T$ be a smooth projective family of surfaces over an integral scheme $T$, and $g$ be an  automorphism of $\calX/T$. Let $g_t$ denote the restriction of $g$ to a fiber $\calX_t = f^{-1}(t)$. Then the function $\phi:t\mapsto \lambda(g_t)$ is lower semi-continuous (i.e. for any real number $r$, the set $\{t\in T:\phi(t)\le r$ is closed).
\end{thm}

One can use this result, for example, when $\pi$ is a family of lattice polarized K3 surfaces. We use this result in the case of a family of Enriques surfaces when a general fiber has no smooth rational curves but the special fibers acquire them. It shows  that the dynamical degree of an automorphism depends on  the Nikulin nodal invariant of the surface (see \cite{Dolgachev1}).

\textbf{Acknowledgement}. This paper owes much to the conversations with  Paul Reschke whose joint work with Bar Roytman makes a similar experiment with automorphisms of K3 surfaces isomorphic to a hypersurface in the product $\bbP^1\times \bbP^1\times \bbP^1$. I am also grateful to Keiji Oguiso and Xun Yu for useful remarks on my talk on this topic during the ''Conference on K3 surfaces and related topics'' held in Seoul on November 16-20, 2015. I thank  the referee whose numerous  critical comments have greatly improved the exposition of the paper. 

\section{Double plane involutions}
\subsection{Enriques surfaces and the lattice $\sfE_{10}$} Let $S$ be an Enriques surface, i.e. a smooth projective algebraic surface with canonical class $K_S$ satisfying $2K_S = 0$ and the second Betti number (computed in \'etale cohomology if  $\bbk\ne \bbC$) equal to 10.  Together with K3 surfaces, abelian surfaces and bielliptic surfaces, Enriques surfaces make the class of algebraic surfaces with numerically trivial canonical class. If $\bbk = \bbC$, the universal cover of an Enriques surface is a K3 surface and $S$ becomes isomorphic to the quotient of a K3 surface by a fixed-point-free involution.

Let us remind some known facts about Enriques surfaces which can be found in many sources (e.g.\cite{Cossec}, or  \cite{Dolgachev1} and references therein). It is known that $\Num(S)$ is an even unimodular lattice of signature $(1,9)$, and, as such, it is isomorphic to the quadratic lattice $\sfE_{10}$ equal to the orthogonal sum of the negative definite even unimodular lattice $\sfE_8$ of rank 8 (defined by the negative of the Cartan matrix of a simple root system of type $E_8$) and the unimodular even rank 2 lattice  $U$ defined by the matrix $\left(\begin{smallmatrix}0&1\\
1&0\end{smallmatrix}\right)$. The lattice $\sfE_{10}$ is called in \cite{Cossec} the \emph{Enriques lattice}. One can choose a basis $(f_1,\ldots,f_9,\delta)$ in $\Num(S)$ formed by isotropic vectors $f_1,\ldots,f_9$ and a vector $\delta$ with 
$$\delta^2 = 10,\  (\delta,f_i) = 3,\  (f_i,f_j) = 1, i\ne j.$$
Together with the vector 
$$f_{10} = 3\delta-f_1-\cdots-f_9,$$
the ordered set $(f_1,\ldots,f_{10})$ forms a  $10$-sequence  of isotropic vectors with $(f_i,f_j) = 1, i\ne j$. We say that the isotropic sequence $(f_1,\ldots,f_{10})$ is \emph{non-degenerate}, if each $f_i$ is equal to the numerical class of a nef divisor $F_i$. In the case of Enriques surfaces this means that the intersection of $F_i$ with  any smooth rational curve is non-negative.  
Under this assumption, $\delta$ is the numerical class of an ample  divisor $\Delta$. The linear system $|\Delta|$ defines a closed embedding of $S$ in $\bbP^5$ with the image a surface of degree 10, called a \emph{Fano model} of $S$.

The orthogonal group of the lattice $\sfE_{10}$ contains a subgroup of index 2 which is generated by reflections $s_{\alpha}:x\mapsto x+(x,\alpha)\alpha$ in vectors $\alpha$ with $\alpha^2 = -2$. This group is a Coxeter group with generators $s_{\alpha_i}$, where 
$$\alpha_0 = \delta-f_1-f_2-f_3,\ \alpha_1 = f_1-f_2,\ldots, \alpha_9 = f_9-f_{10}.$$
Its Coxeter diagram is of $T$-shaped type 
\begin{figure}[h]
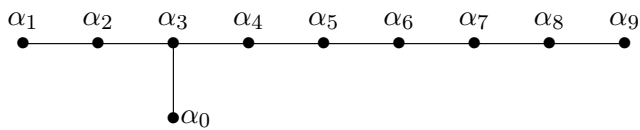

\xy (-40,10)*{};
@={(0,0),(10,0),(20,0),(30,0),(40,0),(50,0),(60,0),(70,0),(80,0),(20,-10)}@@{*{\bullet}};
(0,0)*{};(80,0)*{}**\dir{-};
(20,0)*{};(20,-10)*{}**\dir{-};
(0,3)*{\alpha_1};(10,3)*{\alpha_2};(20,3)*{\alpha_3};(30,3)*{\alpha_4};(40,3)*{\alpha_5};(50,3)*{\alpha_6};(60,3)*{\alpha_7};
(70,3)*{\alpha_8};(80,3)*{\alpha_9};(23,-10)*{\alpha_0};
\endxy
\caption{The Enriques lattice $\sfE_{10}$ and the Coxeter diagram of its reflection group}\label{enriqueslattice}
\end{figure}

\subsection{Elliptic fibrations and double plane involutions}\label{ss2.2}
We have $\Num(X) \cong \Pic(X)/\bbZ K_S$, so there are two choices for $F_i$ if $K_S\ne 0$ and one choice if $K_S = 0$ (the latter may happen only if the characteristic $\cha(\bbk)$ of $\bbk$ is equal to 2). The linear system $|2F_i|$ defines a genus one fibration $p_i:S\to \bbP^1$ on $S$, an elliptic fibration if $\cha(\bbk)\ne 2$.  The fibers $2F_i$ are its  double fibers, the other double fiber is $2F_i'$, where  $F_i'$ is linearly equivalent to $F_i+K_S$.

 The linear system $|2F_i+2F_j|$ of effective divisors linearly equivalent to $2F_i+2F_j$ defines a degree $2$ morphism 
 $$\phi_{ij}:S\to \calD \subset \bbP^4$$
 onto a  del Pezzo surface $\calD$ of degree $4$. If $p\ne 2$, it has four nodes and it is isomorphic to the quotient of $\bbP^1\times \bbP^1$ by an involution with four isolated fixed points. Let $g_{ij}$ be the birational involution of $S$ defined by the deck transformation. Since $S$ is a minimal surface with nef canonical class, it extends to a biregular involution of $S$. We call it a \emph{double plane involution}, a birational equivalent model of the map $\phi_{ij}$ is the original Enriques's double plane construction of Enriques surfaces. 

The natural action of $\Aut(S)$ on $\Num(S)$ defines a homomorphism  
$$\rho:\Aut(S)\to \Or(\sfE_{10}), \ g\mapsto g^*.$$
Its image is contained in the reflection group $W(\sfE_{10})$.   It is known that the kernel of $\rho$ is a finite group of order $2$ or $4$ (see \cite{DolgachevNum},\cite{Mukai},\cite{Mukai2}).  Over $\bbC$ they have been classified in \cite{Mukai2}.  None of them occur in our computations. So we may assume that $\rho$  is injective.

When $S$ is \emph{unnodal}, i.e. it does not contain smooth rational curves\footnote{We call such curves $(-2)$-curves because they are characterized by the property that their self-intersection is equal to $-2$.}, the image $\Aut(S)^*$ of $\rho$ contains a subgroup of finite index of $W(\sfE_{10})$ that coincides with the 2-level congruence subgroup $W(\sfE_{10})(2): = \{\sigma:\half(\sigma(x)-x)\in \Num(S),\  \textrm{for all $x\in \sfE_{10}$}\}$.
It is known that, in this case $g_{ij}$ acts on $\sfE_{10}$ as $-\id_{\sfE_8}\oplus \id_U$ for some orthogonal decomposition $\sfE_{10} =\sfE_8\oplus U$. The subgroup $W(\sfE_{10})(2)$ coincides with the normal closure of any $g_{ij}^*$.

Let $(f_1,\ldots,f_{10})$ be a non-degenerate isotropic $10$-sequence. Consider the degree 2 cover $\phi_{ij}:S\to \calD$ corresponding to a pair 
$(f_i,f_j)$. The map $\phi_{ij}$ blows down common rational irreducible components of fibers of the genus one fibrations $|2F_i|$ and $|2F_j|$. The classes of these components span a negative definite sublattice $\calR_{ij}$ of $\Num(S)$. It is a negative definite lattice spanned by vectors with norm equal to $-2$, the orthogonal sum of root lattices of simple Lie algebras of types  $A_n,D_n,E_n$. The action of the involution $g_{ij}$ on this lattice is given in the following lemma (see \cite{Shimada}, section 3).

\begin{lemma}\label{shimada} Assume $p\ne 2$ and let $X$ be a smooth minimal projective surface of non-negative Kodaira dimension. Let $f:X\to Y$ be a morphism of degree 2 onto a normal surface.  Then any connected  fiber $C = f^{-1}(y)$ over a nonsingular point of $Y$ is a point or the union of $(-2)$-curves whose intersection graph is of type $A_n,D_n,\sfE_n$ as in the following picture. 

$$
\xy (-20,-20)*{};
(-10,0)*{A_n};@={(0,0),(10,0),(20,0),(40,0),(50,0)}@@{*{\bullet}};
(0,0)*{};(25,0)*{}**\dir{-};(35,0)*{};(50,0)*{}**\dir{-};(30,0)*{\ldots};
(0,3)*{a_1};(10,3)*{a_2};(20,3)*{a_3};(40,3)*{a_{n-1}};(50,3)*{a_n};
(-10,-10)*{D_n};@={(0,-10),(10,-10),(20,-10),(40,-10),(50,-10),(10,-17)}@@{*{\bullet}};
(0,-10)*{};(25,-10)*{}**\dir{-};(35,-10)*{};(50,-10)*{}**\dir{-};(30,-10)*{\ldots};
(10,-10)*{};(10,-17)*{}**\dir{-};
(0,-7)*{d_2};(13,-17)*{d_1};(12,-7)*{d_3};(20,-7)*{d_4};(40,-7)*{d_{n-1}};(50,-7)*{d_n};
(-10,-25)*{\sfE_n};@={(0,-25),(10,-25),(20,-25),(30,-25),(50,-25),(60,-25),(20,-32)}@@{*{\bullet}};
(0,-25)*{};(35,-25)*{}**\dir{-};(45,-25)*{};(60,-25)*{}**\dir{-};(40,-25)*{\ldots};
(20,-25)*{};(20,-32)*{}**\dir{-};
(0,-22)*{e_2};(10,-22)*{e_3};(22,-22)*{e_4};(23,-32)*{e_1};(30,-22)*{e_5};(50,-22)*{e_{n-1}};(60,-22)*{\sfE_n};
\endxy
$$
The deck transformation $\sigma$ of $f$ extends to a biregular automorphism of $X$. It acts on the components of $C$ as follows
\begin{itemize}
\item $\sigma(a_i) = a_{n+1-i}, i = 1,\ldots,n$;
\item $\sigma(d_i) = d_i$ if $n$ is even;
\item $\sigma(d_1)=d_2, \sigma(d_i) = d_i, i\ne 1,2$ if $n$ is odd;
\item $\sigma(e_1) = e_1, \sigma(e_i) = e_{8-i}, i\ne 1,$ if $n = 6$;
\item $\sigma(e_i) = e_i$ if $n = 7,8$.
\end{itemize}
\end{lemma}

Let $M_{ij}$ be the sublattice spanned by $f_i,f_j$ and the sublattice 
$\calR_{ij}$. For any $\gamma\in \Num(S)$,  we can write 
\beq\label{new}
\gamma = (\gamma,f_i)f_j+(\gamma,f_j)f_i+r+\gamma^\perp,
\eeq
where $r\in \calR_{ij}$ and $\gamma^\perp\in M_{ij}^\perp$. Since $\gamma^\perp$ is contained in the orthogonal complement of the eigensubspace of $\Num(S)_\bbQ$ with eigenvalue $1$, we have the $g_{ij}^*(\gamma^\perp) = -\gamma^\perp.$ Applying $g_{ij}^*$ to \eqref{new}, we obtain
\beq\label{main}
g_{ij}^*(\gamma) = -\gamma +2(\gamma,f_i)f_j+2(\gamma,f_j)f_i+g_{ij}^*(r)+r.
\eeq
This formula will allow us to compute the action of $g_{ij}$ on $\Num(S)$.

\section{Salem numbers of products of double plane involutions}
\subsection{The first experiment: a general case}In our first experiment,  we assume that the  maps  $\phi_{12},\ldots,\phi_{kk+1}$  are finite morphisms, i.e. do not blow down any curves.  In this case $M_{ij}$ is spanned by $f_i,f_j$, and we get from \eqref{main}
\beq\label{eq2}
g_{ij}^*(f_a) = 2f_i+2f_{j}-f_a, \ a\ne i,j, \quad g_{ij}^*(f_a) = f_a,\  a = i,j.
\eeq
The formula gives the  matrix of $g_{ij}^*$ in the basis $(f_1,\ldots,f_{10})$ of $\Num(S)_\bbQ$. For any $k = 1,\ldots,10$ let
$$c_k: = g_{12}\circ\cdots \circ g_{kk+1}, $$
where $g_{10,11}:= g_{1,10}$. 
Note that the cyclic permutation of   $(g_{12},g_{23},\ldots,g_{10,1})$,  gives a conjugate composition. However, different order on the set of involutions lead sometimes to non-conjugate compositions. 

In this and the following experiment we will consider only the automorphisms $c_k$'s. Using formula \eqref{main}, we compute the matrix of $c_k$ in the basis $(f_1,\ldots,f_{10})$,  find its characteristic polynomial and find its Salem factor and its spectral radius. We check that $c_2^*$ is not hyperbolic and  obtain that  
the Salem polynomials of $c_3^*,\ldots,c_{10}^*$ are equal to 
\begin{eqnarray*}
&&x^4-16x^3+14x^2-16x+1,\\
&&x^2-14x+1,\\
&&x^6-54x^5+63x^4-84x^3+63x^2-54x+1,\\
&&x^6-70x^5-113x^4-148x^3-113x^2-70x+1,\\
&&x^6-186x^5-129x^4-332x^3-129x^2-186x+1,\\
&&x^8-320x^7-548x^6-704x^5-698x^4-704x^3-548x^2-320x+1,\\
&&x^{10}-706x^9+845x^8-1048x^7+1202x^6-1048x^3+845x^2-706x+1,\\
&&x^8-992x^7-1700x^6-1568x^5-1466x^4-1568x^3-1700x^2-992x+1,
\end{eqnarray*}
respectively.

\subsection{The second experiment}
In our second experiment, we assume that each divisor class
$$r_{ii+1} = f_i+f_{i+1}-f_{10-i+1},\ i = 1,2,3,4,$$
is effective and represented by a $(-2)$-curve. For $m > 1$, the curves are disjoint. In the Fano model, they  are smooth rational curves of degree 3.

We have  $M_{ii+1} = \la f_i,f_{i+1},r_{ii+1}\ra$ for $i = 1,\ldots,4$ and $M_{ii+1} = \la f_i,f_{i+1}\ra$ for $i > m$.  
Using formula \eqref{main}, we obtain the following minimal polynomials of $c_3^*,\ldots,c_{10}^*$.

\begin{itemize}
\item $m = 1$:
\begin{eqnarray*}
&&x^4-12x^3+6x^2-12x+1,\\
&&x^2-10x+1,\\
&&x^6-42x^5+31x^4-44x^3+31x^2-42x+1,\\
&&x^6-50x^5-65x^4-92x^3-65x^2-50x+1,\\
&&x^6-142x^5-145x^4-260x^3-145x^2-142x+1,\\
&&x^8-236x^7-316x^6-404x^5-394x^4-404x^3-316x^2-236x+1,\\
&&x^{8}-452x^7+452x^6-892x^5+502x^4-892x^3+452x^2-452x+1,\\
&&x^8-576x^7+44x^6-704x^5-90x^4-704x^3+44x^2-576x+1.
\end{eqnarray*}
\end{itemize}

\begin{itemize}
\item $m = 2$
\begin{eqnarray*}
&&x^2-8x+1,\\
&&x^4-8x^3-6x^2-8x+1,\\
&&x^6-34x^5-5x^4-52x^3-5x^2-34x+1,\\
&&x^6-42x^5-21x^4-68x^3-21x^2-42x+1,\\
&&x^8-120x^7+8x^6-136x^5+46x^4-136x^3+8x^2-120x+1,\\
&&x^6-218x^5-113x^4-300x^3-113x^2-218x+1,\\
&&x^{10}-430x^9+305x^6-192x^7+206x^6-36x^5+206x^4-192x^3+305x^2-430x+1,\\
&&x^{10}-354x^9-231x^8-272x^7-282x^6-28x^5-282x^4-272x^3-231x^2-354x+1.
\end{eqnarray*}
\end{itemize}

\begin{itemize}
\item $m = 3$
\begin{eqnarray*}
&&x^4-5x^3-5x+1,\\
&&x^4-8x^3-2x^2-8x+1,\\
&&x^8-27x^7+26x^6-53x^5+42x^4-53x^3+26x^2-27x+1,\\
&&x^6-35x^5+11x^4-66x^3+11x^2-35x+1,\\
&&x^8-97x^7+146x^6-207x^5+250x^4-207x^3+146x^2-97x+1,\\
&&x^8-173x^7-230x^6-99x^5-22x^4-99x^3-22x^2-99x+1,\\
&&x^{8}-389x^7+186x^6-267x^5-278x^4-267x^3+186x^2-389x+1,\\
&&x^8-291x^7-246x^6-221x^5-214x^4-221x^3-246x^2-291x+1.
\end{eqnarray*}
\end{itemize}

\begin{itemize}
\item $m = 4$
\begin{eqnarray*}
&&x^4-5x^3-5x+1,\\
&&x^6-5x^5-4x^4-12x^3-4x^2-5x+1,\\
&&x^8-21x^7+5x^6-43x^5+4x^4-43x^3+5x^2-21x+1,\\
&&x^6-33x^5-8x^4-60x^3-8x^2-33x+1,\\
&&x^8-91x^7-91x^6-133x^5-124x^4-133x^3-91x^2-91x+1,\\
&&x^8-165x^7+223x^6-59x^5-133x^4-59x^3+223x^2-165x+1,\\
&&x^6-371x^5-62x^4-80x^3-62x^2-371x+1,\\
&&x^{10}-277x^9-104x^8+390x^7-25x^6-546x^5-25x^4+390x^3-104x^2-277x+1.
\end{eqnarray*}
\end{itemize}
 The results of the computations of the dynamical degrees are given in the following table.

\begin{table}[ht]
\centering 
  \begin{tabular}{|c|c|c|c|c|c|c|c|c|}\hline
k/m&3&4&5&6&7&8&9&10 \\ \hline\hline
0&$15.1$&$13.9$&$52.8$&$71.6$&$186.7$&$321.7$&$704.8$&$993.7$\\
$1$&$11.6$&$9.9$&$41.3$&$51.3$&$143.0$&$273.3$&$429.2$&$575.9$\\
$2$&$7.8$&$8.8$&$34.2$&$42.5$&$119.9$&$218.5$&$429.2$&$354.6$\\
$3$&$5.2$&$8.3$&$26.0$&$34.7$&$95.5$&$174.3$&$388.5$&$291.8$\\
$4$&$5.2$&$6.0$&$20.80$&$33.2$&$92.0$&$163.6$&$371.2$&$277.4$	
\\\hline\hline
\end{tabular}
 \end{table}

One may ask whether the configurations of the curves representing the classes 
$r_{kk+1}$ can be realized on an Enriques surface. We omit the details to give the positive answer. The numerical classes of $r_{ii+1}$ define the Nikulin invariant $r(S)$ of the surface $S$ (see \cite{Dolgachev1}, \S 5).
If $\bbk = \bbC$, one can use the Global Torelli Theorem to realize this Nikulin invariant by curves of degree $\le 4$ with respect to the Fano polarization. These curves would realize our $(-2)$-curves $r_{ii+1}$. We omit the details of this rather technical theory.

\section{Enriques surfaces of Hessian type} Here we assume that $\cha(\bbk)\ne 2,3$. 
\subsection{Cubic surfaces and  their Hessian quartic surfaces} All material here is very classical, and, if no explanation of a fact is given, one can find it in many sources, for example, in  \cite{DG}, \cite{DK}, \cite{CAG}.   Let $C$ be a nonsingular cubic surface in $\bbP^3$ given by equation $F = 0$. The determinant  of the Hessian matrix of the second partial derivatives of $F$ is a homogeneous polynomial of degree 4. It defines a quartic surface $H(C)$, the \emph{Hessian surface} of $C$. We assume that  $C$ is  \emph{Sylvester non-degenerate}, i.e. $F$ can be written as a sum of cubes of five linear forms $l_1,\ldots,l_5$ in projective coordinates $t_1,\ldots,t_4$ in $\bbP^3$, no four of which are linearly dependent (see \cite{CAG}, p. 260). By multiplying $l_i$'s by constants, we may assume that 
$$F = \lambda_1l_1^3+\cdots+\lambda_5l_5^3 = 0, \quad l_1+\cdots+l_5 = 0.$$ 
Then, we embed $C$ in $\bbP^4$ with coordinates $(x_1,\ldots,x_5)$ via the map given by 
$x_i = l_i(t_1,\ldots,t_4)$, $i = 1,\ldots,5$. The image is a cubic surface given by  equations
\beq\label{cub}
x_1+\cdots+x_5 = 0, \quad \sum_{i=1}^5\lambda_ix_i^3 = 0.
\eeq
The image of the Hessian surface $H(C)$ of $F$ is given by equations
\beq\label{hess}
\sum_{i=1}^5x_i = \sum_{i=1}\frac{1}{\lambda_i x_i} = 0,
\eeq 
where we understand that the second equation is multiplied by the common denominator to obtain a homogeneous equation of degree 4 (see \cite{CAG}, 9.4.2).

The union of the planes $l_i = 0$ is called the \emph{Sylvester pentahedron}. Its image in $\bbP^4$ is the union of the intersection of the coordinates hyperplanes $x_i = 0$ with the hyperplane $x_1+\cdots+x_5 = 0$. The Hessian surface contains 10 edges $L_{ab}$ of the pentahedron given by equations $x_a = x_b = x_1+\cdots+x_5 = 0$. It also contains its 10 vertices given by $x_i=x_j=x_k = x_1+\cdots+x_5 = 0$. They are ordinary double points $P_{ab}$ on $H(F)$, where 
$\{i,j,k,a,b\} = \{1,2,3,4,5\}$.  A line $L_{ab}$ contains three points $P_{cd}$, where $\{a,b\}\cap \{c,d\} = \emptyset$. Each point $P_{ab}$ is contained in three lines $L_{cd}$. The lines and points form a well-known abstract \emph{Desargues configuration} $(10_3)$ (see \cite{Hilbert}, III, \S 19).

  Figure \ref{sylvester} is the picture of the Sylvester pentahedron with vertices $P_{ab}$ and edges $L_{ab}$. 

\begin{figure}[ht]
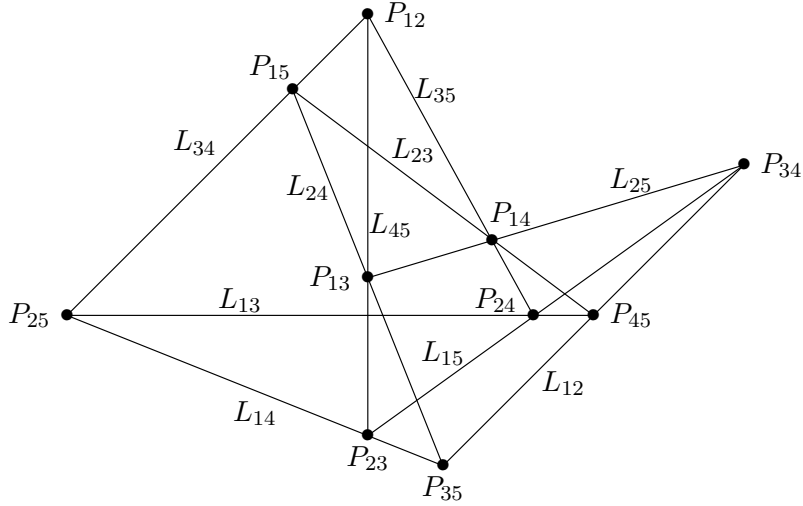

\xy (-40,-25)*{};
(0,0)*{};(40,40)*{}**\dir{-};(0,0)*{};(70,0)*{}**\dir{-};(0,0)*{};(50,-20)*{}**\dir{-};
(50,-20)*{};(90,20)*{}**\dir{-};(40,40)*{};(40,-16)*{}**\dir{-};(50,-20)*{};(30,30)*{}**\dir{-};
(40,40)*{};(62,0)*{}**\dir{-};(40,-16)*{};(90,20)*{}**\dir{-};(70,0)*{};(30,30)*{}**\dir{-};(90,20)*{};(40,5)*{}**\dir{-};
@={(0,0),(30,30),(40,40),(70,0),(90,20),(50,-20),(40,-16),(62,0),(40,5),(56.5,10)}@@{*{\bullet}};
(-5,0)*{P_{25}};(27,33)*{P_{15}};(45,40)*{P_{12}};(95,20)*{P_{34}};(75,0)*{P_{45}};(50,-23)*{P_{35}};
(40,-19)*{P_{23}};(35,5)*{P_{13}};(59,13.5)*{P_{14}};(57,2)*{P_{24}};
(17,23)*{L_{34}};(25,-13)*{L_{14}};(23,2)*{L_{13}};(66,-9)*{L_{12}};(43,12)*{L_{45}};(49,30)*{L_{35}};
(32,17)*{L_{24}};(75,18)*{L_{25}};(50,-5)*{L_{15}};(46,22)*{L_{23}};
\endxy
\caption{Sylvester pentahedron}\label{sylvester}
\end{figure}

\subsection{The Cremona involution of the Hessian surface and Enriques surfaces}
The  birational transformation
\beq\label{cremona}
\Phi_{\st}:(x_1,x_2,x_3,x_4,x_5) \mapsto \big(\frac{1}{\lambda_1x_1},\frac{1}{\lambda_2x_2},\frac{1}{\lambda_3x_3},\frac{1}{\lambda_4x_4},\frac{1}{\lambda_5x_5}\big)
\eeq
of $\bbP^4$ satisfies $\Phi_{\st}^2 = \id_{\bbP^4}$. It is equal to the \emph{standard Cremona involution} of $\bbP^4$ after one scales the coordinates. The fixed points of $\Phi_{\st}$ (in the domain of the definition) have coordinates $(\pm \sqrt{\lambda_1}^{-1},\ldots,\pm \sqrt{\lambda_5}^{-1})$. We additionally assume that none of them lies on the Hessian surface.  Under our assumption,  the birational involution $\Phi_{\st}$ defines a fixed-point-free involution $\tau$ on a minimal nonsingular model $X$ of the Hessian surface $H(C)$ with the quotient isomorphic to an  Enriques surface. We say that  such an Enriques surface is of \emph{Hessian type}. Let
$$\pi:X\to S = X/(\tau).$$
be the projection map. It is the K3-cover of $S$.

Let $N_{ab}$ be the exceptional curve over $P_{ab}$ on $X$ and let $T_{ab}$ be the proper transform of the edge $L_{ab}$ on $X$.
It follows from the definition of the birational transformation $\Phi_{\st}$ that the involution $\tau$ of $X$ sends $N_{ab}$ to $T_{ab}$. Let $h$ be the pre-image on $X$ of the class of a hyperplane section of $H(C)$. We can represent it by  a coordinate hyperplane section of $H(C)$. It is the union of four edges of the Sylvester pentahedron. It is easy to see from Figure \ref{sylvester}  that 
\beq\label{h}
h = \sum_{i\in \{a,b\}}T_{ab}+ \sum_{i\not\in \{c,d\}}N_{cd}, \quad i = 1,2,3,4,5.
\eeq
Applying the involution $\tau$, we obtain
$$h+\tau^*(h) = \sum_{\{a,b\}}(T_{ab}+ N_{ab}).$$
Here and later we identify the divisor class of any smooth rational curve on $X$ or on $S$ with the curve. 

 Let $U_{ab}$ denote the image of the curves $N_{ab}+T_{ab}$ on the Enriques surface.  The intersection graph of these 10 curves is the famous \emph{Petersen graph} given in Figure \ref{peter}.  Let $\Delta$ be the sum of the curves $U_{ab}$. We have
 $$\pi^*(\Delta) = h+\tau^*(h).$$
Then it is easily checked that $\Delta^2 = 10$, and $\Delta\cdot U_{ab} = 1$.

 \begin{figure}[th]
\begin{center}
\includegraphics[scale=.2]{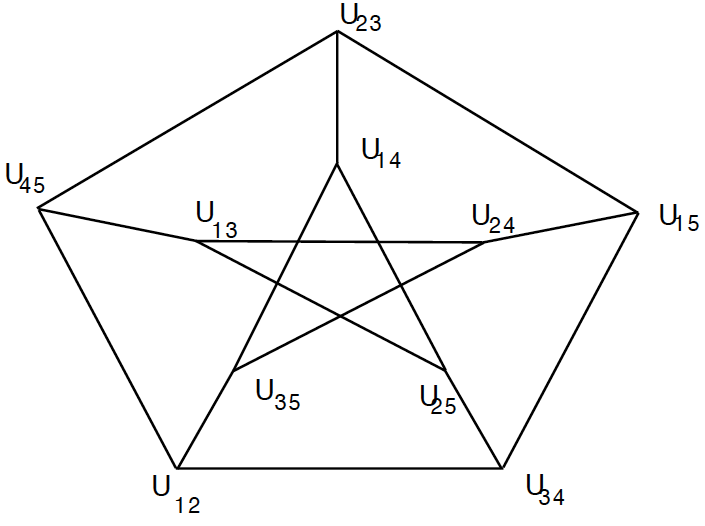}
\caption{Petersen graph}\label{peter}
\end{center}
\end{figure}

\subsection{Elliptic pencils}
Consider the pencil of planes containing an edge $L_{ab}$. A general plane from this pencil cuts out on $H(C)$ the union of the line $L_{ab}$ and a plane cubic curve.  The preimage of the pencil of residual cubic curves on $X$ is a pencil  of elliptic curves on $X$ defined by the linear system 
\beq\label{gab}
|G_{ab}|: = |h-T_{ab}-\sum_{\{c,d\}\cap \{a,b\} = \emptyset} N_{cd}|.
\eeq
If we represent $h$ by a plane containing $L_{ab}$ and the point $P_{ab}$, the residual curve is a plane cubic with a double point at $P_{ab}$. Let $K_{ab}$ be its proper transform on $X$. Another representative of $h$ is a plane tangent to $H(C)$ along the line $L_{ab}$. The residual curve is a conic intersecting $L_{ab}$ with multiplicity 2. Let $K_{ab}'$ be its proper transform on $X$. We obtain
\beq\label{gab2}
|G_{ab}| = |N_{ab}+K_{ab}| = |T_{ab}+K_{ab}'|.
\eeq
It follows from the definition of the standard Cremona involution that $\Phi_{\st}$  preserves the pencil of hyperplanes through any edge of the pentahedron. Thus the elliptic pencil $|G_{ab}|$ is invariant with respect to $\tau$, and we get 
$$\tau^*(K_{ab}) = K_{ab}'.$$
Being $\tau$-invariant, the elliptic pencil $|G_{ab}|$ descends to an elliptic fibration $|\bar{G}_{ab}|$ on $S$. Note that, for a general divisor $D\in |\bar{G}_{ab}|$, we have  $\pi^*(D) = G+G'$, where $G$ and $G'$ are disjoint members of $|G_{ab}|$.  As we explained in subsection \ref{ss2.2}, we can write $|G_{ab}|$  as  $|2F_{ab}| = |2F_{ab}'|$, where the numerical classes $f_{ab},f_{ab}'$ of $F_{ab},F_{ab}'$ are primitive nef isotropic vectors in $\Num(S)$.
 We immediately check that $G_{ab}\cdot G_{cd} = 2$ for different subscripts, hence $f_{ab}\cdot f_{cd} = 1$ and the classes $f_{ab}$  form an isotropic $10$-sequence. It follows from \eqref{gab} that 
$$2\sum_{\{ab\}}G_{ab} = \sum_{\{ab\}}G_{ab}+\sum_{\{ab\}}\tau^*(G_{ab}) = 10(h+\tau^*(h))-4\sum_{\{a,b\}}(N_{ab}+T_{ab}) = 6(h+\tau^*(h)).$$
Thus,
$$\sum_{\{a,b\}}f_{ab} = 3\delta,$$
where $\delta$ is the numerical class of the Fano polarization 
\beq\label{delta1}
\Delta = \sum_{\{ab\}}U_{ab}.
\eeq
It follow that each curve $U_{ab}$ becomes a line in the Fano model of $S$.  Equation \eqref{gab2} shows that each pencil  $|2F_{ab}|$ has a reducible member $U_{ab}+C_{ab},$ where $C_{ab}$ is the image of the curves $K_{ab},K_{ab}'$ on $S$.  Another reducible member of this fibration is the sum of curves $U_{cd}$, where $\#\{a,b\}\cap \{c,d\} = 1$. It is of  type $I_6$ in Kodaira's classification of singular fibers of relatively minimal elliptic fibrations (see \cite{Cossec}, Chapter III, \S 1).

\subsection{Projection involutions}
Recall that the projection map of a quartic surface $H$ with a double point $p$ with center at $p$ defines  a rational  map $H\da \bbP^2$ of degree $2$. Its deck transformation extends to a biregular involution $\sigma_p$ of  a minimal nonsingular model  of $H$. Let $\sigma_{ab}$ be such a transformation of the minimal nonsingular model $X$ of the Hessian surface $H(C)$ defined by the projection from the point $P_{ab}$.

The following lemma  follows easily from the definition of the projection map.

\begin{lemma}\label{action2}  Let $\sigma_{12}$ be the involution of $X$ defined by the projection map from the node $P_{12}$. Then
\begin{eqnarray*}
\sigma_{12}^*(N_{12}) &=&K_{12}'\ \text{if $K_{12}$ is irreducible}, \\
      &=&N_{12}\ \text{otherwise},\\
\sigma_{12}^*(T_{12}) &=&K_{12}\ \text{if $K_{12}$ is irreducible}, \\
      &=&T_{12}\  \text{otherwise},\\
\sigma_{12}^*(N_{13}) &=& N_{23}, \quad \quad \sigma_{12}^*(T_{13})  =  T_{23},\\
\sigma_{12}^*(N_{14}) &=& N_{24}, \quad \quad \sigma_{12}^*(T_{14})  =  T_{24},\\
\sigma_{12}^*(N_{15}) &=& N_{25}, \quad \quad \sigma_{12}^*(T_{15})  =  T_{25}\\
\sigma_{12}^*(N_{34}) &=& N_{34}, \quad \quad \sigma_{12}^*(T_{34})  =  T_{34}\\
\sigma_{12}^*(N_{35}) &=& N_{35}, \quad \quad \sigma_{12}^*(T_{35})  =  T_{35}\\
\sigma_{12}^*(N_{45}) &=& N_{45}, \quad \quad \sigma_{12}^*(T_{45})  =  T_{45},\\
\end{eqnarray*}      
\end{lemma}

It follows that $\sigma_{ab}^*$ acts on the curves $N_{cd},T_{cd}, \{c,d\}\ne \{a,b\},$ via the transposition $(ab)$ applied to the subscript indices. Moreover, if $K_{12}$ is reducible then it acts as a transposition on all curves $N_{ab}$ and $T_{ab}$.

\begin{cor}\label{action} The projection involutions $\sigma_{ab}$ commute with the involution $\tau$ and descend to involutions $h_{ab}$ of the Enriques surface $S = X/(\tau)$. The involution $h_{12}$ acts on the curves $U_{ab}$ as follows:
\begin{eqnarray*}
h_{12}^*(U_{13})&=&U_{23},\\
h_{12}^*(U_{14})&=&U_{24},\\
h_{12}^*(U_{15})&=&U_{25},\\
h_{12}^*(U_{34})&=&U_{34},\\
h_{12}^*(U_{35})&=&U_{35},\\
h_{12}^*(U_{45})&=&U_{34}.
\end{eqnarray*}
$$
h_{12}^*(U_{12}) = \begin{cases}
     C_{12} & \text{if}\ C_{12}\ \text{is irreducible}, \\
      U_{12}& \text{otherwise}.
\end{cases}$$
\end{cor}

Assume $C_{ab}$ is irreducible. Let  
$$\alpha_{ab} = f_{ab}-U_{ab}\in \Num(S).$$
We have 
$$
\alpha_{ab}^2 = -2.$$

We denote by $t_{ab}$ the  transformation of $\Pic(S)_\bbQ$ that permutes the basis $(U_{ij})$ via the 
 transposition $(ab)$ of $\{1,2,3,4,5\}$. Note that, in general, it is not realized by an automorphism of the surface.  
 
 \begin{cor} Assume $C_{ab}$ is irreducible. Then $h_{ab}^*$ acts as the composition of the reflection 
$r_{\alpha_{ab}}$ and the transformation $t_{ab}$. If $C_{ab}$ is reducible, then it acts as the transposition $(ab)$.
\end{cor}

\begin{proof} If $C_{ab}$ is reducible, the assertion follows from the previous corollary. Assume $C_{ab}$ is irreducible. By definition of the reflection transformation, we have 
$$r_{\alpha_{ab}}(x) = x+(x,\alpha_{ab})\alpha_{ab}.$$
It follows from the inspection of the Petersen graph in Figure \ref{peter} that each $U_{cd}$ with $\{c,d\}\cap \{a,b\} = \emptyset$ intersects a fiber of the elliptic fibration $|2F_{ab}|$ with multiplicity 2. It also intersects $U_{ab}$ with multiplicity 1. Thus, it must intersect $C_{ab}$ with multiplicity 1. This implies 
that $r_{\alpha_{ab}}(U_{cd}) = U_{cd}$, hence $h_{ab}^*(U_{cd}) = (r_{\alpha_{ab}}\circ t_{ab})(U_{cd})$. If 
$\#\{c,d\}\cap \{a,b\} = 1$, then $U_{cd}\cdot U_{ab} = U_{cd}\cdot C_{ab} = 0$ and 
$h_{ab}^*(U_{cd}) = (r_{\alpha_{ab}}\circ t_{ab})(U_{cd})$ again. Finally, we have $U_{ab}\cdot \alpha_{ab} = 2$, hence
$$r_{\alpha_{ab}}(U_{ab}) = U_{ab}+2(f_{ab}-U_{ab}) = 2f_{ab}-U_{ab} = C_{ab}, \quad t_{ab}(U_{ab}) = U_{ab}.$$
Hence, again $h_{ab}^*(U_{ab}) = (r_{\alpha_{ab}}\circ t_{ab})(U_{ab}).$ 
\end{proof}

Note that 
$$\alpha_{ab}\cdot \alpha_{cd} = \begin{cases}
      1& \text{if}\ \{a,b\}\cap \{c,d\} \ne \emptyset, \\
      0& \text{otherwise}.
\end{cases}
$$
This implies that the group generated by the reflections $r_{\alpha_{ab}}$ is isomorphic to the Coxeter group $\Gamma$ with the Coxeter diagram  equal to the \emph{anti-Petersen graph}.\footnote{The terminology is due to S. Mukai.} It is a $6$-valent regular graph with 10 vertices and 30 edges. It is obtained from the complete graph $K(10)$ with 10 vertices by deleting the edges from the Petersen graph.  The group $\frakS_5$ acts on the graph and hence acts on $\Gamma$ by outer automorphisms.

\begin{cor} Let $S$ be a general Enriques surface of Hesse type. The group $G$ of automorphisms of $S$ generated by the projection involutions $h_{ab}$ is isomorphic to a subgroup of $\Gamma\rtimes \frakS_5$, where $\Gamma$ is the subgroup of $W(\sfE_{10})$ isomorphic to the Coxeter group with the anti-Petersen graph as its Coxeter graph.
\end{cor} 

Recall from the previous section that each pair of isotropic vectors $f_{ab}$ and $f_{cd}$ defines an involution $g_{ab,cd}$ of $S$. The following lemma shows that the group of automorphisms generated by these involutions is contained in the group $G$  generated by the projection involutions $h_{ab}$.

\begin{lemma} We have
$$g_{ab,cd} = h_{ab}\circ h_{cd}\  \textrm{if}\ \{a,b\}\cap \{c,d\} = \emptyset,$$
and
$$g_{ab,bc} = h_{de}\  \textrm{if}\  \{a,b,c,d,e\} = \{1,2,3,4,5\}.$$
\end{lemma}

\begin{proof} One checks that the two pencils $|2F_{ab}|$ and $|2F_{cd}|$ have four common components if $\{a,b\}\cap \{c,d\} = \emptyset$ and 5 common components otherwise. For example, $|2F_{12}|$ and $|2F_{34}|$ have common 
components $U_{13},U_{14},U_{23},U_{24}$ (they correspond to subsets which share an element with both subsets $\{a,b\}$ and $\{c,d\}$). On the other hand, $|2F_{12}|$ and $|2F_{23}|$  have common components $U_{12},U_{23},U_{13},U_{24},U_{25}$.\footnote{In \cite{DK}, p. 3034, the second case was overlooked.} In 
the former case, the sublattice $\calR_{ab,cd}$ is isomorphic to the orthogonal sum of the root lattices $A_2\oplus A_2$, and, in the latter case, it is isomorphic to $A_1\oplus A_1\oplus A_3$. Assume first that 
$\{a,b\}\cap \{c,d\} = \emptyset$. Without loss of generality, we may assume that $\{a,b\} = \{1,2\}$ and $\{c,d\} = \{3,4\}$. Applying Lemma \ref{shimada}, we find that $g_{12,34}$ acts on $\calR_{12,34}$ by switching
$U_{13}$ with $U_{24}$ (they span the root lattice $A_2$) and $U_{23}$ with $U_{14}$ (they span the other connected component of the root lattice). It follows from Corollary \ref{action} that $h_{12}\circ h_{34}$
acts the same on $\calR_{12,34}$. It follows from formula \eqref{main} that they act the same on $\Num(S)$, hence coincide. 

Now let us assume that $\{a,b\} = \{1,2\}$ and $\{c,d\} = \{2,3\}$. The curves $U_{25},U_{13},U_{23}$ span a 
sublattice of $\calR_{12,23}$ of type $A_3$, the curves $U_{12}$ and $U_{23}$ span the sublattices of 
type $A_1$. Applying Lemma \ref{shimada}, we find that $g_{12,23}$ acts on $\calR_{12,23}$ by switching 
$U_{25}$ with $U_{24}$ and leaving other curves unchanged.  The involution $h_{45}$ does the same job. This proves the assertion.
\end{proof}

Now we are in business and can compute the actions of the compositions of involutions $h_{ab}$ on $\Num(S)$. Applying Corollary \ref{action} and Lemma \ref{eckardt}, we obtain that $h_{ab}^*$  acts as the composition of a reflection $s_{\alpha_{ab}}$ and the transposition $t_{ab}$.

  We will restrict ourselves with elements $g\in G$ which are products of $N$ different $h_{ab}$'s. The support of such a word in the generators defines a subgraph of the Petersen graph of  cardinality $N$. Different graphs may define different conjugacy classes and the same  subgraph may  correspond to different conjugacy classes. The following list gives  the smallest Salem numbers which we were able to find in this way. \footnote{According to \cite{Shimada} 100 hours 
of computer computations using random choice of automorphisms suggests that
4.33064... is the minimal Salem number realized by an automorphism of a general Enriques surface of Hessian type.}

\begin{table}[ht]
\centering 
  \begin{tabular}{|c|c|c|c|}\hline
$d$&minimal polynomial&$\lambda(g)$&automorphism \\ \hline\hline
$2$&$x^2-5x+1$&$4,7912\ldots$&$(1,2,3,4,5,6,7)$\\
$4$&$x^4-4x^3-2x^2-4x+1$&$4.3306\ldots$&$(2,6,1,3)$\\ 
$6$&$x^6-6x^5+6x^4-6x^3+6x^2-6x+1$&$5,0015\ldots$&$(1,2,3,4,5,6,8)$\\
$8$&$x^8-7x^7+3x^4-9x^5+8x^4-9x^3+3x^2-7x+1$&$6,7309\ldots$&$(2,3,1,8,9,10)$\\ \notag
$10$&$x^{10}-17x^9-6x^8-10x^7+5x^6-10x^5+\ldots$&$17,3775\ldots$&$(1,3,10,5,7,9,8)$\\ \hline\hline 
\end{tabular}
 \caption{Salem numbers for a general Enriques surface of Hesse type }\label{table1}
 \end{table}

Here we order the pairs $(a,b)$ as $(12,13,14,15,23,24,25,34,35,45)$ and give the sequence of indices in the product of involutions $h_{ab}$. 

\subsection{Eckardt points and specializations}
Recall that an \emph{Eckardt point} in a nonsingular cubic surface $C$ is a point $x\in C$ such that the tangent plane  at $x$ intersects $C$  along the union of three lines intersecting at $x$. A general cubic surface does not have such points. Moreover, if $C$ is Sylvester non-degenerate and given by equation \eqref{hess},  then the number of Eckardt points is equal to the number of distinct unordered pairs $\{i,j\}$ such that $\lambda_i =\lambda_j$ (see \cite{CAG}, Example 9.1.25). Thus, the possible number $r$ of Eckardt points is equal to $0,1,3,6$ or $10$.

\begin{lemma}\label{eckardt} The following assertions are equivalent:
\begin{itemize}
\item[(i)] the residual conic $K_{ab}$ is reducible;
\item[(ii)]   $\lambda_a = \lambda_b$;
\item[(iii)] there exists a  plane tangent to $H(C)$ along  the edge $L_{ab}$ that contains the point $P_{ab}$; 
\item[(iv)] the point $P_{ab}$ is an Eckardt point of the cubic surface $C$.
 \end{itemize}
\end{lemma}

\begin{proof} For simplicity of notation, we assume that $\{a,b\} = \{1,2\}$. It is easy to see that the equation of the  plane tangent to $H(C)$ along the line $L_{12}$ is given by equations  $\lambda_1x_1+\lambda_2x_2 =  x_1+\cdots+x_5 = 0$.  The residual conic is given by the additional equation  
$$\lambda_4\lambda_5x_4x_5+\lambda_3\lambda_5x_3x_5+\lambda_4\lambda_5x_4x_5 = 0.$$
One checks that  it is singular, and hence reducible,  if and only if $\lambda_1 = \lambda_2$. The singular point in this case is the  point $P_{12}$. This proves the equivalence of the first three statements. Using equation \eqref{cub} we find that (ii) implies that the point $P_{12} = (1,-1,0,0,0)$ is an Eckardt point of the cubic surface $C$. Conversely, if $P_{12}$ is an Eckardt point, we easily find that $\lambda_1 = \lambda_2$.

\end{proof}

A Sylvester non-degenerate cubic surface $C$ with $10$ Eckardt points is isomorphic to the Clebsch diagonal surface (see \cite{CAG}, 9.5.4). The equation of the Hesse surface $H(C)$ is 
\beq\label{penc}
x_1+x_2+x_3+x_4+x_5 = \frac{1}{x_1}+\frac{1}{x_2}+\frac{1}{x_3}+\frac{1}{x_4}+\frac{1}{x_5} = 0.
\eeq
If $\cha(\bbk) \ne 3,5$, the involution $\Phi_{\st}$ has no fixed points on $H(C)$ and the quotient is an Enriques surface with the automorphism group isomorphic to $\frakS_5$ (see \cite{DG}, Remark 2.4). It is of Type VI (not IV as was erroneously claimed in loc. cit.) in in Kondo's list of Enriques surfaces with finite automorphism group \cite{Kondo}. If $\cha(\bbk) = 3$\footnote{In this case the Clebsch diagonal surface must be given by the equation 
$s_1 = s_3 = 0$, where $s_k$ are elementary symmetric polynomials in $5$ variables.} (resp. $\cha(\bbk) = 5$), the involution has five (resp. one) fixed points, and the minimal resolution of the quotient is a Coble surface. We will discuss Coble surfaces in the next section.

We assume that $\cha(\bbk) > 5$ and consider the case when the number $r$ of Eckardt points is equal to  $ 6$.  Thus, we may assume that the Hessian surface $H$ is given by equations:
\beq\label{pencil}
x_1+x_2+x_3+x_4+x_5 = \frac{1}{x_1}+\frac{1}{x_2}+\frac{1}{x_3}+\frac{1}{x_4}+\frac{1}{t x_5} = 0.
\eeq
 It is immediately checked that the involution $\tau$ has no fixed points if and only if $t\ne 1/4,1/16$. These cases lead to Coble surfaces and will be considered in the next section. 

The six nodes $P_{ab}$ with $a,b\in \{1,2,3,4\}$ of $H$ are the Eckardt points of the cubic surface. It follows from Lemma \ref{eckardt} that, for each such pair, the conic $C_{ab}$ splits into the union of two lines passing through the point $P_{ab}$. For example, the components of $C_{12}$ are given by equations 
$$x_1+x_2 =  x_3-t\omega x_5 = x_4-t\omega^2x_5 = 0,$$
where $\omega$ is the third  root of $1$ different from 1. 

 As before, we have involutions $ h_{ab}$ defined by the projections from the nodes. It follows from Corollary \ref{action}  that
 $h_{ab}^*,a<b\le 4,$ coincides with the transposition $t_{ab}$. 
Computing the action of products of different involutions, we can substantially decrease the spectral radii. Our best results are given in Table \ref{table2}.

\begin{table}[ht]
\centering 
  \begin{tabular}{|c|c|c|c|}\hline
$d$&minimal polynomial&$\lambda(g)$&automorphism \\ \hline\hline
$2$&$x^2-4x+1$&$3.7320\ldots$&$(7,8,10,1,4)$\\
$4$&$x^4-x^3-2x^2-x+1$&$2.0810\ldots$&$(2,5,8,7,10)$\\ 
$6$&$x^6-4x^5-x^4-4x^3-x^2-4x+1$&$4.4480\ldots$&$(6,8,7,1,9,4)$\\
$8$&$x^8-4x^7+4x^6-5x^5+4x^4-5x^3+4x^2-4x+1$&$3,1473\ldots$&$(7,8,9,2)$\\ \notag
$10$&$x^{10}-6x^9-7x^8-9x^7-6x^6-10x^5-6x^4-9x^3-7x^2-6x+1$&$ 7.1715\ldots$&$(1,5,8,4,7,5,4,10)$\\ \hline\hline 
\end{tabular}
 \caption{Salem numbers for a special pencil of  Enriques surface of Hesse type }\label{table2}
 \end{table}

Here we order the pairs as $(12,13,14,15,23,24,25,34,35,45)$ and give the sequence of subscripts in the product of the $h_{ab}$'s. Note that the Salem number of degree 2 is the smallest possible and the Salem number of degree 4 is the third smallest possible.
 
Comparing the Tables \ref{table1} and \ref{table2}, we see a great improvement in our search of small Salem numbers realized by automorphisms of an Enriques surface.

\section{Hyperbolic automorphisms of Coble surfaces}
\subsection{Coble rational surfaces}  Suppose a K3 surface  $X$ together with a fixed-point-free involution $\tau$ degenerates to a pair $(X_0,\tau_0)$ where $X_0$ is a K3 surface and $\tau_0$ has  a smooth rational curve as its locus of fixed points. The formula for the canonical class of the double cover $X_0\to X_0/(\tau_0)$ shows that the quotient surface $V = X_0/(\tau_0)$ is a rational surface such that $|-K_V| = \emptyset$ but $|-2K_V| = \{C\}$, where   $C$ is a smooth rational curve with $C^2 = -4$, the branch curve of the projection $X_0\to V$. 

A smooth projective surface $V$ with the property that $|-K_V| = \emptyset$ but $|-2K_V| \ne \emptyset$ is called a \emph{Coble surface}. The classification of Coble surfaces can be found in \cite{DZ}. One can show that $V$ must be a rational surface and $h^0(-2K_V) = 1$, i.e. $|-2K_V|$ consists of an isolated positive divisor $C$. In the case when $C$ is a smooth, its connected components $C_1,\ldots,C_s$ are smooth rational curves with self-intersection equal to $-4$. The double cover branched along $C$ is a K3 surface $X_0$, and the pair $(X_0,\tau_0)$, where $\tau_0$ is the covering involution can be obtained as a degeneration of a pair $(X,\tau)$, where $X$ is a K3 surface and $\tau$ is its fixed-point-free involution. We will be dealing only with Coble surfaces such that $C\in |-2K_V|$ is smooth (they are called \emph{Coble surfaces of K3 type} in \cite{DZ}).

A Coble surface of K3 type is a basic rational surface, i.e. it admits a birational morphism $\pi:V\to \bbP^2$. The image of the curve $C = C_1+\cdots+C_s$ in $\bbP^2$ belongs to $|-2K_{\bbP^2}|$, hence it is a plane curve $B$ of degree 6. The images of the components $C_i$ are  its irreducible components or points.  We have $K_V^2 = -s$, so $V$ is obtained from $\bbP^2$ by blowing up $9+s$ double points of $B$, some of them may be infinitely near points. 

Assume that $s = 1$. In this case $B$ is an irreducible rational curve of degree 6, and $V$ is obtained by blowing up its 10 double points $p_1,\ldots,p_{10}$. We say that  a  Coble surface is \emph{unnodal} if  the nodes of $B$ are in general position in the following sense:no infinitely near points, no three points on a line, no six points are on a conic, no plane cubics pass through 8 points one of them being a double point, no plane quartic curve passes through the points with one of them a triple point. These conditions guarantee that the surface does not contain  $(-2)$-curves  (see \cite{CD}, Theorem 3.2). The rational plane sextics with this property were studied by A. Coble  who showed that the orbit  of such curves under the group of birational transformations consists of only finitely many projectively non-equivalent curves \cite{Coble}.

Assume that $B$ is an irreducible curve of degree 6 with 10 nodes none of them are infinitely near. The orthogonal complement of $K_V^\perp$ in $\Num(V)$ of the canonical class $K_V$ is a quadratic lattice isomorphic to $\sfE_{10}$. It has a basis formed by 
the vectors $e_0-e_1-e_2-e_3,e_1-e_2,\ldots,e_9-e_{10}$, where $e_0$ is the class of the pre-image of a line in the plane and $e_1,\ldots,e_{10}$ are the classes of the exceptional curves of the blow-up. Since $-K_V = 3e_0-(e_1+\cdots+e_{10})$, a different basis is formed by the classes $(\delta,f_1,\ldots,f_9)$, where 
\beq\label{delta2}
\delta = 10e_0-3(e_1+\cdots+e_{10}),\eeq
 and  $f_i = 3e_0-(e_1+\cdots +e_{10})+e_i, i = 1,\ldots,9$. We have 
\beq\label{delta22}
3\delta = f_1+\cdots+f_{10},
\eeq
where $f_{10} = 3e_0-(e_1+\cdots+e_9)$ and $\delta^2 = 10$. The classes $f_i$ represent the classes of the proper transforms on $V$ of cubic curves passing through nine of the double points of $B$.
All of this is in a complete analogy with isotropic $10$-sequence $(f_1,\ldots,f_{10})$ on an Enriques surface and its Fano polarization $\delta$ which we dealt with in the previous sections. Moreover, the linear system 
$|2f_i+2f_j|$ defines a degree 2 map onto a 4-nodal quartic del Pezzo surface
$$\phi_{ij}:V\to \calD,$$
as in the case of Enriques surfaces. The difference here is that the map is never finite  since it blows down the curve $C\in |-2K_V|$ to one of the four ordinary double points of $\calD$. One can show that the deck transformation of $\phi_{ij}$ defines a biregular automorphism $g_{ij}$ of $V$. Its action on the lattice $K_V^\perp$ is similar to the action of the covering involution in  the case of Enriques surfaces. It acts identically on the sublattice $M_{ij}$ spanned by $f_i,f_j$ and the invariant part of the sublattice spanned by common irreducible components of the elliptic fibrations $|2F_|$ and $|2F_j|$. It acts as the minus identity on the orthogonal complement of $M_{ij}$ in $K_V^\perp$.  This allows us to consider the group generated by $g_{ij}$ and compute the dynamical degrees of a hyperbolic automorphism from this group. They completely agree with ones we considered for an Enriques surface. 

Assume that $V$ is an unnodal Coble surface. It is known (the fact is essentially due to A. Coble) that the normal closure of any $g_{ij}^*$ in $W(K_V^\perp) \cong W(\sfE_{10})$ coincides with the $2$-level congruence subgroup $W(\sfE_{10})(2)$ (see \cite{CD}). This should be compared with the same result in the case of Enriques surfaces.

Assume that the isotropic sequence $(f_1,\ldots,f_{10})$ is non-degenerate. If  $R$ is a common irreducible component of $|2f_i|$ and $|2f_j|$, then intersecting with $\delta$, we find that $(\delta,R)\le 4$. Since $(K_V,R) = 0$, this is equivalent to that  $R$ is a proper inverse transform of a plane curve of degree $\le 4$. It could be a line 
passing through three points $p_i,p_j,p_k$, a conic through 6 points $p_i$, a cubic passing  through eight points, one of them is its singular point, or a quartic passing through all points, two of them are its singular points. They are represented by the respective classes in $\sfE_{10}$ of types
$\delta-f_i-f_j-f_k,\  f_i+f_j+f_k+f_l-\delta, \ f_i+f_j-f_k,\  \delta-2f_i.$
These are exactly classes of rational smooth curves in $\sfE_{10}$ which we used in the previous sections. This allows us to realize  explicit  examples of degenerations of a Coble surface in a much easier and more explicit way comparing to the case of Enriques surfaces. The dynamical degrees of automorphisms $c_k = g_{1,2}\cdots g_{m,m+1}$ coincide with ones computed for Enriques surfaces. 

\subsection{Coble surfaces of Hessian type}
Assume that the Hessian quartic surface $H$ of a Sylvester non-degenerate cubic surface acquires an additional singular point and the involution $\Phi_{st}$ fixes this point. Then the quotient of a minimal resolution $X$ of $H$ is a Coble surface. We call it a \emph{Coble surface of Hessian type}.

Consider the pencil of Hessian surfaces \eqref{pencil}. 
When $t = \frac{1}{4}$, we obtain the equation of the Hessian of the Cayley $4$-nodal cubic surface 
$$x_1x_2x_3+x_1x_3x_4+x_1x_2x_4+x_2x_3x_4 = 0.$$
Its group of automorphisms is isomorphic to $\frakS_4$ which acts by permuting the coordinates. The singular points of the cubic surface form the orbit of the point $(1,0,0,0)$. 

When $t = \frac{1}{16}$, we obtain the equation of the Hessian surface of a cubic surface with one ordinary double point and 6 Eckardt points (see \cite{DG}, Lemma 4.1).

One easily checks that any  member $H_t$ of the pencil \eqref{pencil} has 10 singular points which form two orbits of $\frakS_4$ (acting by permuting the first four coordinates) of the points $(1,0,0,0,0)$ and $(1,-1,0,0,0)$.  The lines and the points are the edges and the vertices of the Sylvester pentahedron.

The surface $H_{\frac{1}{4}}$ has four more ordinary double points forming the orbit of the point $(1,1,1,-1,-2)$. The surface $H_{\frac{1}{16}}$ has only one more singular point $(1,1,1,1,-4)$. None of the additional singular points lie on the edges.\footnote{The Hessian surfaces $H_{\frac{1}{4}}$ and $H_{\frac{1}{16}}$ were studied in detail in \cite{DG}, in particular, the authors compute  the Picard lattices of their minimal resolutions.}

Let $\Phi_{\st}$ be the  Cremona transformation defined in \eqref{cremona}. 
It leaves invariant each member $H_t$ of the pencil with parameter value $t$ and defines a biregular involution on a minimal resolution $X_t$ of $H_t$. Its fixed points are the pre-images on $X_t$ of the points $(x_1,x_2,x_3,x_4,x_5)$ with $t x_5^2 = 1$ and $x_i^2 =  1, 1\le i\ne 4$. We have 
$x_5 = -(x_1+\cdots+x_4) = \pm 4, \pm 2,$ or $0$. Then we obtain that $\Phi_{\st}$ has no fixed points on $H_t$ unless $t = \frac{1}{4}$ or $t = \frac{1}{16}$. In this case, the fixed points are the additional ordinary double points. Thus, we obtain  that 
$$V_{\frac{1}{16}}:= X_{\frac{1}{16}}/(\Phi_{\st}), \quad V_{\frac{1}{4}} = X_{\frac{1}{4}}/(\Phi_{\st})$$
are Coble surfaces with $K_{V_{\frac{1}{16}}}^2 = -1$ and $K_{V_{\frac{1}{4}}}^2 = -13$.

The curves $N_{ab}$ and $T_{cd}$ which are defined for $X_t$ with general $t$ survive for the special values $t = \frac{1}{4}, \frac{1}{16}$. They form the same Desargues configuration exhibited in Figure \ref{sylvester}.  The involution $\Phi_{\st}$ switches the curves $N_{ab}$ and $T_{ab}$.

\subsection{Coble surfaces and the Desargues Theorem}
Recall from projective geometry that a \emph{Desargues configuration} of lines in the projective plane is formed by six sides of two perspective triangles, three lines joining the perspective vertices and the line joining the three intersection points of these three pairs. It follows from Desargues Theorem (e.g. \cite{CAG}, Theorem 2.1.11) that 
any configuration of lines and points forming the abstract configuration $(10_3)$ comes from two perspective triangles. We can see the picture of this configuration from Figure \ref{sylvester}. Here the two triangles  are the triangles with vertices $P_{15},P_{13},P_{14}$ and $P_{25},P_{35},P_{45}$. 

Consider the orbit of the curves $N_{ab}+T_{ab}$ on the quotient Coble surface $V_t, t = \frac{1}{4}, \frac{1}{16}$. It is a $(-2)$-curve $U_{ab}$. Let $\pi:V_t\to \bbP^2$ be the blowing down morphism.

\begin{prop} Let $V$ be a Coble surface of Hessian type. Assume that $|-2K_V|$ is represented by an irreducible smooth rational curve $C$ with self-intersection $-4$.  The image of the ten curves $U_{ab}$ in the plane is the set of 10 lines forming a Desargues configuration of lines. The double points of this configuration are the double points of the sextic curve $B$.
\end{prop}

\begin{proof} Let $H$ be the Hessian surface and $X$ be its minimal resolution. The curve $C\in |-2K_V|$ is the exceptional curve of the resolution $X_t\to H_t$ over its unique node that is different from the ten modes $P_{ab}$. We know that the image of the curves $C_i$ in the plane are the irreducible components of the sextic curve $B$. Since the additional double point of $H_t$ does not lie on an edge, we obtain that $U_{ab}\cap C = \emptyset$. Thus the images $\ell_{ab} = \pi(U_{ab})$ of the curves $U_{ab}$ intersect $B$ only at its double points.  Recall that $U_{ab}\cdot f_i  = 1$ for exactly three values of $i$ and zero otherwise. Also it follows from \eqref{delta1} (which is still valid for Coble surfaces) that $U_{ab}\cdot \delta = 1$. Since $3\delta = \sum_{i=1}^{10}f_i$, it follows from \eqref{delta2}, that the class $e_0$ of a line intersects $U_{ab}$ with multiplicity one,  hence $\ell_i$ is a line. 
Since the curves $U_{ab}$ form the Petersen graph from Figure 2, their images form a Desargues configuration of lines. 
\end{proof}

\begin{remark} The beautiful fact that the vertices of a Desargues configuration of lines in the projective plane are double points of a plane curve of degree $6$ is due to J. Thas \cite{Thas}. The moduli space of Coble surfaces of Hessian type is isomorphic to the moduli space of nodal Sylvester non-degenerate cubic surfaces. It contains an open dense subset parameterizing one-nodal cubic surfaces. The six lines passing through the node define an unordered set of six points on the exceptional curve at this point isomorphic to $\bbP^1$, hence they define a hyperelliptic curve of genus $2$. Conversely, a set of six unordered points on $\bbP^1$ define, via the Veronese map,  six points on a conic. Their blow-up is isomorphic to a one-nodal cubic surface. In this way we see that the moduli space of Coble surfaces of Hessian type contains an open dense subset isomorphic to the $3$-dimensional moduli space of genus 2 curves. On the other hand, it is known that the moduli space of Desargues configurations in the plane also contains an open dense subset naturally isomorphic to the moduli space of genus 2 curves \cite{Avritzer}. Thus the  argument from the previous proof gives an algebraic geometrical proof of Thas's result for a general Desargues configuration.
\end{remark}

\subsection{Coble surface $V_{\frac{1}{16}}$}
Assume $t = \frac{1}{16}$. The surface $V_{\frac{1}{16}}$ is the Coble surface associated to the Hessian surface of a cubic surface with six Eckardt points and one ordinary double  point. It is obtained by blowing up 10 double points of an irreducible plane sextic $B$. The nodes are the vertices of a Desargues configuration of 10 lines, the images of the curves $U_{ab}$ in the plane. We also know that the Hessian surface $H_{\frac{1}{16}}$ and hence $V_{\frac{1}{16}}$ admits $\frakS_4$ as its group of automorphisms. This group permutes the curves $U_{ab}$ and hence, considered as a group of birational transformations of the plane, it leaves invariant the Desargues configuration of ten lines. In particular,  it permutes the exceptional curves representing the classes $e_1,\ldots, e_{10}$. It follows from formula \eqref{delta1} that it leaves $\Delta = \sum U_{ab}$ invariant. Appying formula \eqref{delta2}, we see that it leaves invariant the class $e_0$ of a line in the plane. Thus the symmetry group $\frakS_4$ comes from a group of projective transformations leaving invariant the curve $B$. Since no non-identical projective transformation can fix a curve of degree $> 1$ point-wise, we see that $B$ admits $\frakS_4$ as its group of projective automorphisms. According to R. Winger \cite{Winger}, there are two projective equivalence classes of such irreducible rational sextics (\emph{octahedral sextics}). One of them has six cusps among its double points. It cannot be ours since the group $\frakS_4$ permutes the divisor classes $e_i$ and hence has one orbit on the set of double points of the sextic.  So,  our sextic is another  one that is given by parametric equation
$$(x,y,z) = \bigl(u^6-5u^2v^4,uv(v^4-5u^4),v(u^5+v^5)\bigr).$$
Its set of ten nodes contains 6 biflecnodes (i.e. nodes whose branches are inflection points) which are vertices of a complete quadrilateral of lines which together with the remaining four nodes form a \emph{quadrangle-quadrilateral configuration} (see \cite{VY}, Chapter II). Together,  four sides of the quadrilateral and six sides of a quadrangle form a Desargues configuration of lines. This agrees with our partition of the curves $U_{ab}$ into two sets corresponding to indices $a,b\le 4$ and $(a,b) = (a,5)$.

\subsection{Coble surface $V_{\frac{1}{4}}$}
Assume now that $t = \frac{1}{4}$. In this case $K_{V_{\frac{1}{4}}}^\perp$ is generated over $\bbQ$ by the classes of the curves $U_{ab}$ and the classes $C_2-C_1,C_3-C_1,C_4-C_1$. Its rank is equal to $13$. One checks that the hyperplane section $x_a+x_b = 0, 1\le a < b\le 4,$ intersects  the Hessian along the union of the edge $x_a=x_b = 0$ and a line $\ell_{ab} = x_a+x_b = x_c-x_d = 0, \{a,b\}\cap \{k,l\} =\emptyset$ taken with multiplicity 2. The Cremona involution $\Phi_{\st}$ leaves the lines $\ell_{ab}$ invariant. The hyperplane section spanned by the line $\ell_{ab}$ and the edge $x_a = x_b = 0$ is tangent to $H_{1/4}$ along the  edge and the line. The pencil of hyperplanes through the edge defines an elliptic fibration on $X_{\frac{1}{4}}$ with two reducible fibers of type $I_6$, one reducible fiber of type $I_0^*$ and one reducible fiber of type $I_4$. The latter two fibers correspond to the planes $x_a+x_b = 0$ and $x_a-x_b = 0$, respectively. Each plane contains a pair of new singular points. The fiber of  type $I_0^*$  is equal to the divisor $D_1 = 2R_0+R_1+R_2+R_3+R_4$, where $R_0$ is the proper inverse transform of the line $\ell_{ab}$, the curves $R_1,R_2$ are the exceptional curves over the new singular points lying on $\ell_{ab}$ which are pointwisely fixed by $\Phi_{\st}$ and the curves $R_3,R_4$ are the curves $N_{cd}$ and $T_{cd}$. The image of $D_1$ on the Coble surface $V_{\frac{1}{4}}$ is the divisor $4E_0+2E_1+E_2+E_3$, where $E_0^2 = -1, E_1^2 = -2, E_2^2 = E_3^2 = -4$ and 
$E_0\cdot E_i = 1, i= 1,2,3, E_i\cdot E_j = 0, 1\le i < j = 3$. Its image on the relatively minimal elliptic surface is a fiber of Kodaira's type III.

\begin{figure}[ht]
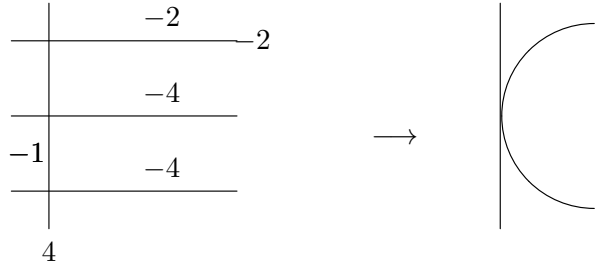

\xy (-40,45)*{};
(0,10)*{};(0,40)*{}**\dir{-};(-5,15)*{};(25,15)*{}**\dir{-};
(-5,25)*{};(25,25)*{}**\dir{-};(60,10)*{};(60,40)*{}**\dir{-};(-5,35)*{};(25,35)*{}**\dir{-};
(72.5,25)*\cir<35pt>{l_r};
(46,22)*{\longrightarrow};
(-3,20)*{-1};(15,38)*{-2};(15,28)*{-4};(15,18)*{-4};(0,7)*{4};
(-3,20)*{-1};(27,35)*{-2};
\endxy
\caption{The image of the fiber of type $I_0^*$ on $V_{\frac{1}{4}}$}
\label{fig4}
\end{figure}

The  fiber of type $I_4$ is a divisor $D_2 = A_1+A_2+A_3+A_4$ forming a quadrangle of $(-2)$-curves. Two opposite sides  $A_1,A_2$ are the proper transforms of the components of the residual cubic curve. The curves $A_3,A_4$ are the exceptional curves over two new singular points. The image of $D_2$ in $V_{\frac{1}{4}}$ is a quadrangle of four smooth rational curves as pictured on  Figure \ref{fig5}. Its image on the relatively minimal elliptic surface is a fiber of  type $I_2$.

\begin{figure}[ht]
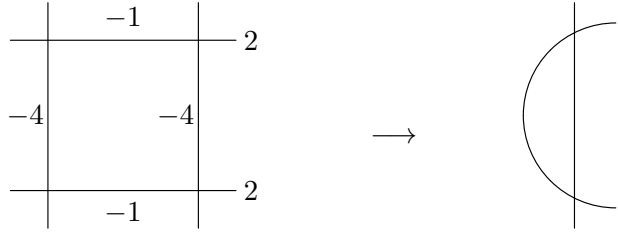

\xy (-40,45)*{};
(0,10)*{};(0,40)*{}**\dir{-};(-5,15)*{};(25,15)*{}**\dir{-};
(70,10)*{};(70,40)*{}**\dir{-};(-5,35)*{};(25,35)*{}**\dir{-};
(75.5,25)*\cir<35pt>{l_r};
(46,22)*{\longrightarrow};(20,10)*{};(20,40)*{}**\dir{-};
(-3,25)*{-4};(10,38)*{-1};(17,25)*{-4};(10,12)*{-1};(27,15)*{2};(27,35)*{2};
\endxy
\caption{The image of the fiber of type $I_4$ on $V_{\frac{1}{4}}$}
\label{fig5}
\end{figure}
Let $f:V_{\frac{1}{4}}\to S$ be the blowing down of the curves $E_0,E_1$ and two $(-1)$-curves in the second fiber. The image of the elliptic pencil on $S$ defines a relatively minimal elliptic fibration on the surface $S$. It has reducible fibers of Kodaira's types $I_6,I_2$ and III.

Let us  give an explicit construction of $V_{\frac{1}{4}}$ as the blow-up of $13$ points in the plane. Consider the union of two lines $\ell_1$ and $\ell_2$ and two conics $K_1$ and $K_2$ satisfying the following properties
\begin{itemize}
\item[(i)] The lines $\ell_1, \ell_2$ are tangent to the conic $K_1$ at points $p_1,p_2$;
\item[(ii)] $K_1$ is tangent to $K_2$ at two points $p_3,p_4$;
\item[(iii)] the lines $\la p_1,p_3\ra$ and $\la p_1,p_4\ra$ through the points $p_1,p_3$ and through the points $p_1,p_4$ intersect $\ell_2$ at the points $p_5,p_6$ lying on the line $\ell_2$.
\end{itemize}

This configuration of lines exists and is unique up to projective equivalence. In fact, let us choose the projective coordinates such that the points $p_1,p_3,p_4$ are the points $[1,0,0], [0,1,0]$ and $[0,0,1]$. Then, after scaling the coordinates we can write the equation of $K_1$ in the form $xy+yz+xz = 0$ and the equation of $K_2$ in the form $xy+xz+yz+ax^2 = 0$.  The equation of $\ell_1$ becomes $y+z = 0$. The lines $\la p_1,p_3\ra$ and $\la p_1,p_4\ra$ become the lines $y = 0$ and $z = 0$. They intersect $K_2$ at the points $[1,-a,0]$ and $[1,0,-a]$. The  line through these points has the equation $ax+y+z = 0$. The condition that this line is tangent   to $K_1$ at some point $p_2$ is $a = 4$. This makes $\ell_2$  and $K_2$ unique. So the union of the two lines and the two conics is given by the equation
$$B:= (y+z)((y+z+4x)(xy+yz+xz)(xy+yz+xz+4x^2) = 0.$$
To get the Coble surface, we first blow-up the points  $p_0 = \ell_1\cap \ell_2, p_1,p_3,p_4,p_5,p_6$ and infinitely near points 
$p_1',p_3',p_4'$ to $p_1,p_3,p_4$ corresponding to the tangent directions of tangent lines at the conics. They are the base points of the pencil of cubics spanned by $\ell_1+K_2$ and $\ell_2+K_1$. The corresponding minimal elliptic surface contains three reducible fibers of types III coming from $\ell_2+K_1$, of type $I_2$ coming from $\ell_1+K_2$ and a fiber of type $I_6$ coming from the coordinate lines $x = 0, y= 0, z= 0$. Then we blow-up the singular points of the first two fibers as in above to get the Coble surface.

\subsection{Dynamical degrees of automorphisms of $V_{\frac{1}{4}}$ and $V_{\frac{1}{16}}$}
Let $G$ be as before, the group generated by  the projection involutions  $h_{ab}$. Obviously, the sublattice $N$ of $\Num(V_{\frac{1}{4}})$ generated by the curves $C_1,\ldots,C_4$ is invariant with respect to $\Aut(V_{\frac{1}{4}})$. The group $G$ acts on it by permuting the classes of $C_i$. Thus, all eigenvectors of elements  $g^*, g\in G$ with real eigenvalue $\lambda > 1$ are contained in  $L = (N^\perp)_\bbR$. Since the sublattice generated by the curves $U_{ab}$ is contained in $L$, we can compute the spectral radii of elements $g$ in the basis formed by these curves. It is obvious that the action on these curves has not changed comparing to $V_t$ with general $t$. 

To sum up, we have proved the following.

\begin{thm} Let $(H_t)_{t\ne 0}$ be the family of Hessian surfaces given by equation \eqref{hess}. Let 
$(S_t')_{t\ne 0}$ be the family of surfaces obtained as the quotients $H_t$ by the Cremona involution $\Phi_{\st}$. 
For $t\ne \frac{1}{4},\frac{1}{16}$, a minimal smooth model $S_t$ of $S_t'$ is an Enriques surface. The surface $S_1$ is an Enriques surface with the automorphism group isomorphic to $\frakS_5$ (see  subsection 4.5). For $t = \frac{1}{4},\frac{1}{16}$, a minimal resolution of singularities of   $S_t'$ is a Coble surface $S_t$. For any $t\ne 1$,  the group $G_t$ of automorphisms of $H_t$ generated by the deck transformations of the projections from  the nodes of $H_t$ is isomorphic to the semi-direct product $G = \UC(4)\rtimes \frakS_4$, where $\UC(4)$ is the free product of four groups of order $2$.\footnote{It is proved on \cite{MO}, Theorem 1 that this group is the whole group of  automorphisms of the surface.}  The group $G_t$ commutes with the Cremona involution $\Phi_{\st}$ and descends to a subgroup $\bar{G}_t$ of $\Aut(S_t)$. For any 
$g\in \bar{G_t}$, the spectral radius of $g^*$ acting on $\Num(S_t)_\bbR$ is independent of  $t\ne 1$.
\end{thm}

\begin{remark} Note that there is no smooth family $f:\calS\to T$ of surfaces which include both Enriques and Coble surfaces. One of the reasons is that their Euler-Poincar\'e characteristics are different. However, there exists a smooth family $f:\calX\to T$ of K3-surfaces admitting an involution $\tau\in \Aut(\calX/T)$ such that the quotient of a general fiber is an Enriques surface and the quotient of special fibers are Coble surfaces. In this case, the locus of fixed points of $\tau$ is of relative dimension one and the quotient total family $\calX/(\tau)$ is singular along this locus. So, the family $\bar{\calX} = \calX/(\tau)\to T$ has smooth fibers but not a flat (and hence not smooth) family. By blowing up the singular locus we obtain a flat family with special fibers isomorphic to the union of a copy of a Coble surface and $s$ copies of minimal ruled surfaces (a \emph{flower pot} degeneration of Enriques surfaces, see \cite{Morrison}).\end{remark}

\section{Questions and comments}
One of the goals of our computer experiments was to find a hyperbolic automorphism $g$ of an Enriques surface of  small as possible dynamical degree $\lambda(g)$. It follows from \cite{Oguiso} that the first four smallest Salem numbers of degree $10$ cannot be realized by an automorphism of an Enriques surface. His proof can be used to eliminate such numbers as possible dynamical degrees of automorphisms of a Coble surface.  So, the natural question is the following.

\begin{question} What is the smallest Salem number larger than one of given degree $d \ge 4$  realized by an automorphism of an Enriques surface or a Coble surface?
\end{question}

Obviously, the possible degree does not exceed 10 for an Enriques surface. The same is true for a Coble surface since we can restrict  the action of an automorphism to  the sublattice of rank 10 orthogonal to the classes of all irreducible component of the effective anti-bicanonical curve. 

The smallest Salem numbers  in our computations can be found in  Table \ref{table2}. Note that we had realized the second smallest Salem number of degree 2 and the smallest number in degree $4$.

I am sure that these are not optimal results. Note that any Salem number that can be realized as the spectral radius of an element of a Coxeter group with the $T$-shaped Coxeter diagram of type $T_{2,3,n}, n \ge 7$, occurs as the dynamical degree of an automorphism of a rational surface \cite{Uehara}. However, the group of automorphisms of the rational surface realizing these numbers for $n > 7$ is expected to be small, in most cases, an infinite cyclic group.  

Let $f:\calX\to T$ be a smooth family of projective surfaces over an  integral scheme $T$ of finite type  over $\bbk$. We say that it is a complete smooth family if it cannot be extended to a smooth family $\calX'\to C'$ where $C$ is a proper open subset of an integral scheme $C'$. Suppose $G$ is a group of  automorphisms of $\calX/T$ and let $g_\eta$ be the restriction of $g\in G$ to the generic geometric fiber of $f$.  Assume that $g_\eta$ is hyperbolic, i.e. $\lambda(g_\eta) > 1$. Let  $U = \{t\in T:\lambda(g_t)< \lambda(g_\eta)\}$. We say that the family as above is \emph{equi-hyperbolic} if for any $t\in U$ we have $\lambda(g_t) = 1$. In subsection 4.5 we have introduced a family of Enriques surfaces over an algebraically closed field $\bbk$ of characteristic $p\ne 2,3,5$ parameterized by $T = \bbP^1\setminus \{0,\infty,\frac{1}{4},\frac{1}{16}\}$ together with a group of automorphisms $G$ isomorphic to $\UC(4)\rtimes \frakS_4$ such that for all $t\in T, t\ne 1$, any hyperbolic element $g\in G$ restricts to a hyperbolic automorphism $g_t$ whose dynamical degree does not depend on $t$ and $\lambda(g_1) = 1$. Applying the theory of degenerations of K3 and Enriques surfaces, one can show that the family is a smooth complete family.  This is an example of a equi-hyperbolic family of automorphisms of Enriques surfaces. By taking the K3-cover we get an example of an equi-hyperbolic family of automorphisms of K3 surfaces. To exclude the trivial cases of  constant or isotrivial families, we assume additionally, that $U\ne \emptyset$.  

\begin{question} What is the largest possible dimension of a equi-hyperbolic family of automorphisms of algebraic surfaces.
\end{question}

 So far, we have considered the dynamical degree of a single automorphism of an algebraic surface $X$. One may also look at a discrete infinite group $G$ of automorphisms. We identify $G$ with its image $\Gamma$ in the group of isometries of the hyperbolic space $\bbH^n$ associated to a real linear quadratic space $V$ of dimension $n+1$ of signature $(1,n)$ (in our case  $V = \Num(X)_\bbR$). It is a discrete group of isometries of $\bbH^n$. For any  point $x\in \bbH^n$ one defines the Poincar\'e series 
 $$P_s(x,x) = \sum_{\gamma\in \Gamma}e^{-sd(\gamma(x),x)},$$
 where $d(x,y)$ is the hyperbolic distance between the points $x,y$. The \emph{critical exponent} $\delta(\Gamma)$ is defined to be the infimum of the set $\{s\in \bbR:P_s(x,y) < \infty\}$. We have $\cosh d(x,y) = (x,y)$, where $(x,y)$ is the inner product in the Minkowski space $V$.  Suppose that $\Gamma$ is the cyclic group generated by $g^*$, where $g$ is a hyperbolic automorphism of dynamical degree $\lambda(g) > 1$.  If we take $x\in \bbH^n$ to represent the numerical class $h$ of an ample divisor, then  
 $$P_s(h,h) = \sum_{n\ge 0}e^{-s\cosh^{-1}((g^*)^n(h),h))} \sim  \sum_{n\in \bbZ}\frac{1}{(g^*)^n(h),h)^s},$$
 and since $ \lambda(g) =\lim_{n\to \infty}((g^*)^n(h),h))^{1/n} > 1$, the series converges for all $s > 0$. In particular, the critical exponent is equal to zero. On the other hand, if $\Gamma$ is geometrically finite (one of the equivalent definitions of this is that $\Gamma$ admits a fundamental polyhedron with finitely many sides), and non-elementary (i.e. does not contain a subgroup of finite index isomorphic to an abelian group of finite rank) the critical exponent is positive and coincides with the \emph{Hausdorff dimension} of the limit set $\Lambda(\Gamma)$  and the Poincar\'e series diverges at $\delta_\Gamma$ \cite{Sullivan}. 
 
For any  positive real $r$  and any ample class $h$ on $X$, set 
$$N_{X,h}(r) = \#\{g\in G:(g^*(h), h) \le r\}.$$
One can show that in the case when $\Gamma$ is non-elementary geometrically finite discrete group, there exists an asymptotic expansion of this function of $r$ whose leading term is 
equal to $c_{\Gamma,h}r^{\delta_{\Gamma}}$ (see \cite{Dolgachev2}).

 It is known that $0< \delta_\Gamma \le n-1$ for non-elementary geometrically discrete  group $\Gamma$ of isometries of $\bbH^n$ and the equality $\delta_\Gamma = n-1$ holds if and only if $\Gamma$ is of finite covolume. It follows from the Global Torelli Theorem for K3 surfaces that the group $\Aut(X)$ of automorphisms $G$ of a complex algebraic K3 surface 
 $X$ is  finitely generated (see \cite{Sterk}) and it is of finite covolume if and only if $X$ does not contain smooth rational curves. In fact, the proof shows that the subgroup $G$  of the orthogonal group $\Or(\Num(X))$ generated 
 by $\Aut(X)^*:= \{g^*,g\in \Aut(X)\}$ and reflections in $(-2)$-curves is of finite index. It follows that 
 $\Aut(X)^*$ is always a geometrically finite discrete group of isometries of the hyperbolic space (one finds a fundamental polyhedron with finitely many sides by throwing away the faces in the fundamental domain of $G$ corresponding to reflections in $(-2)$-curve).\footnote{This remark is due to V. Nikulin.} Thus, if $G$ is non-elementary and $X$ acquires a smooth rational curve,  we have $0<\delta_{\Gamma}<\rho-2$, where $\rho = \rank \Num(X)$. In particular, $\delta_\Gamma$ drops when $X$ acquires a smooth rational curve.

\begin{question} In notation of Theorem \ref{xie}, where we replace $g$ by a non-elementary group $\Gamma$, what is the behavior of the   function 
 $s\mapsto \delta_{\Gamma_s}$?  
\end{question}

The Hausdorff dimension $\delta_\Gamma$ is notoriously difficult to compute, and numerical computations are known only in a few cases (see \cite{Dolgachev2}).

\begin{question} What is $\delta_\Gamma$ for the group of automorphisms of a K3, an Enriques surface or a Coble surface of Hessian type generated by the projection involutions?
\end{question}


\begin{thebibliography}{20}
\bibitem{Avritzer} D. Avritzer, H. Lange, \textit{Curves of genus 2 and Desargues configurations}. Adv. Geom. {\bf 2} (2002),  259--280. 

\bibitem{Cantat} S. Cantat, \textit{Sur les groupes de transformations birationnelles des surfaces}.  Ann. of Math. (2) {\bf 174} (2011),  299--340.

\bibitem{Cantat2} S. Cantat, \textit{Dynamics of automorphisms of compact complex surfaces}. Frontiers in complex dynamics, 463--514, 
Princeton Math. Ser., 51, Princeton Univ. Press, Princeton, NJ, 2014. 

\bibitem{CD} S. Cantat, I. Dolgachev, 
\textit{Rational surfaces with a large group of automorphisms}. 
J. Amer. Math. Soc. {\bf 25} (2012),  863--905. 

\bibitem{Coble} A. Coble,  \textit{The Ten Nodes of the Rational Sextic and of the Cayley Symmetroid}. Amer. J. Math. {\bf 41} (1919),  243Ð-265.

\bibitem{Cossec} F. Cossec, I. Dolgachev, \textit{Enriques surfaces. I}. Progress in Mathematics, 76. Birkh\"auser Boston, Inc., Boston, MA, 1989.

\bibitem{DG} E. Dardanelli, B.  van Geemen, 
\textit{Hessians and the moduli space of cubic surfaces}.  Algebraic geometry, 17--36, 
Contemp. Math., 422, Amer. Math. Soc., Providence, RI, 2007. 

\bibitem{DZ} I.  Dolgachev, De-Qi Zhang, \textit{Coble rational surfaces}. Amer. J. Math. {\bf 123} (2001),  79--114.


\bibitem{DK} I. Dolgachev,  J.  Keum, \textit{Birational automorphisms of quartic Hessian surfaces}. Trans. Amer. Math. Soc. {\bf 354} (2002),  3031--3057.


\bibitem{CAG} I. Dolgachev, \textit{Classical algebraic geometry. A modern view}. Cambridge University Press, Cambridge, 2012.

\bibitem{Dolgachev1} I. Dolgachev,  \textit{A brief introduction to Enriques surfaces}, Development of Moduli Theory-Kyoto 2013, Adv. Study in Pure Math. Math. Soc. 69, Math. Soc. Japan,  2016, pp. 1--32.


\bibitem{DolgachevNum} I. Dolgachev, \textit{Numerical automorphisms of Enriques surfaces in arbitrary characteristic}, Arithmetic and geometry of K3 surfaces and Calabi-Yau threefold, ed. R. Lazu, M. Sch\"utt, N. Yui, Fields Institute Communications, vol. 67, Springer. 2013, pp.267Ð-284.

\bibitem{Dolgachev2} I. Dolgachev,  \textit{Orbital counting of curves on algebraic surfaces and sphere packings}, K3 surfaces and their modules, ed. C. Faber, G. Farkas, and G. van der Geer, Progress in Math. 315, Birkha\"user, 2016, pp. 17--54.

\bibitem{Gizatullin} M. Gizatullin, \textit{Rational G-surfaces. } Izv. Akad. Nauk SSSR Ser. Mat.  {\bf 44} (1980),  110Ð-144, 239.

\bibitem{Hilbert} D. Hilbert, S. Cohn-Vossen, \textit{Anschauliche Geometrie}, Berlin, Springer, 1932.

\bibitem{Kondo} S. Kond\={o}, \textit{Enriques surfaces with finite automorphism group}, 	Japan J. Math. {\bf 12} (1986), 192--282.

\bibitem{McMul1} C. McMullen, \textit{Coxeter groups, Salem numbers and the Hilbert metric}. Publ. Math. Inst. Hautes \'Etudes Sci. {\bf 95} (2002), 151--183.



\bibitem{McMul3} C. McMullen, \textit{Dynamics on blowups of the projective plane}. Publ. Math. Inst. Hautes ƒtudes Sci. {\bf 105} (2007), 49Ð-89.

\bibitem{McMul4} C. McMullen, \textit{Automorphisms of projective K3 surfaces with minimum entropy},  Invent. Math. {\bf 203} (2016), no. 1, 17--215.
\bibitem{Morrison} D. Morrison, \textit{Semistable degenerations of Enriques' and hyperelliptic surfaces}. Duke Math. J. {\bf 48} (1981),  197--249. 
\bibitem{Mukai} S. Mukai, Y. Namikawa, \textit{Automorphisms of Enriques surfaces which act trivially on the cohomology groups}. Invent. Math. {\bf 77} (1984), 383Ð-397.

\bibitem{Mukai2} S. Mukai, \textit{Numerically trivial involutions of Kummer type of an Enriques surface}. Kyoto J. Math. {\bf 50} (2010),  889Ð-902.

\bibitem{MO} S. Mukai, H. Ohashi, \textit{The automorphism groups of Enriques surfaces covered by symmetric quartic surfaces}. Recent advances in algebraic geometry, 307--320, London Math. Soc. Lecture Note Ser., 417, Cambridge Univ. Press, Cambridge, 2015.

\bibitem{Oguiso} K. Oguiso, \textit{The third smallest Salem number in automorphisms of K3 surfaces}. Algebraic geometry in East Asia--Seoul 2008, 331--360, Adv. Stud. Pure Math., 60, Math. Soc. Japan, Tokyo, 2010. 



\bibitem{Shimada} I. Shimada, \textit{Automorphisms of supersingular K3 surfaces and Salem polynomials}. 
Experimental Math. {\bf 25} (2016),  389--398.

\bibitem{Shimada2} I. Shimada,	 \textit{On an Enriques surface associated with a quartic Hessian surface}, arXiv:1701.00580.

\bibitem{Sterk} H. Sterk, \emph{Finiteness results for algebraic K3 surfaces}. Math. Z. {\bf 189} (1985),  507--513.  

\bibitem{Sullivan} D. Sullivan, \textit{Entropy, Hausdorff measures old and new, and limit sets of geometrically finite Kleinian groups}. Acta Math. {\bf 153} (1984),  259--277.

\bibitem{Thas} J.  Thas, \textit{A rational sextic associated with a Desargues configuration}. Geom. Dedicata {\bf 51} (1994),  163--180.

\bibitem{Uehara} T. Uehara, \textit{Rational surface automorphisms with positive entropy}. Ann. Inst. Fourier (Grenoble) {\bf 66} (2016),  377--432. 

\bibitem{VY} O. Veblen,  J.W. Young, \emph{Projective geometry}. Vol. 1. Blaisdell Publishing Co. Ginn and Co.. New York-Toronto-London, 1965.

\bibitem{Winger} R. Winger, \textit{Self-Projective Rational Sextics}. Amer. J. Math. {\bf 38} (1916), no. 1, 45--56.

\bibitem{Xie} J. Xie, \textit{Periodic points of birational transformations on projective surfaces}. Duke Math. J. {\bf 164} (2015),  903--932.


\end{thebibliography}
\end{document}